\documentclass[11pt]{amsart}
\usepackage{amssymb,mathrsfs,stmaryrd}
\usepackage[all]{xy}
\usepackage{graphicx}
\usepackage[mathcal]{euscript}
\usepackage{verbatim}
\usepackage{tikz}
     \usetikzlibrary{positioning}
\usepackage{subcaption}

\usepackage{algorithm}
\usepackage[noend]{algpseudocode}

\renewcommand{\and}{\qquad\text{and}\qquad}

\usepackage{mathtools}
\usepackage{manfnt}

\usepackage{color}

\definecolor{jade}{rgb}{0.10, 0.56, 0.42}
\definecolor{cerise}{rgb}{0.87, 0.19, 0.39}
\usetikzlibrary{arrows.meta}
\tikzset{>={latex[width=3mm,length=3mm]}}

\usetikzlibrary{arrows,backgrounds, fit, calc, positioning}
\usepackage{multirow}

\theoremstyle{definition}
\newtheorem*{thm*}{Theorem}
\newtheorem{thm}{Theorem}[section]
\newtheorem{definition}[thm]{Definition}

\newtheorem{lem}[thm]{Lemma}
\newtheorem{prop}[thm]{Proposition}

\theoremstyle{remark}
\newtheorem{rmk}[thm]{Remark}
\newtheorem{example}[thm]{Example}

\newcommand{\N}{\Z_{>0}}
\newcommand{\Z}{\mathbb{Z}}

\newcommand{\C}{\mathbb{C}}
\newcommand{\Cs}{\mathbb{C}-\{0\}}
\newcommand{\pp}{\mathbb{P}^1}

\newcommand{\Ad}{\mathrm{Ad}}

\newcommand{\orb}{\mathscr{O}}

\newcommand{\rmd}{r'}

\newcommand{\ftype}{\mathscr{A}}
\DeclareMathOperator{\M}{\mathcal{M}}
\DeclareMathOperator{\Cat}{\mathcal{C}}
\newcommand{\bfA}{\mathbf{A}}
\newcommand{\bfO}{\mathbf{O}}

\DeclareMathOperator{\fgl}{\mathfrak{gl}}

\DeclareMathOperator{\GL}{GL}

\newcommand{\fb}{\mathfrak{b}}
\newcommand{\fu}{\mathfrak{u}}

\newcommand{\iwa}{\mathfrak{i}}

\newcommand{\gc}{\nabla}

\DeclareMathOperator{\Res}{\mathrm{Res}}

\DeclareMathOperator{\Lie}{\mathrm{Lie}}

\DeclareMathOperator{\partmult}{\mathrm{m}}
\DeclareMathOperator{\partsupp}{\mathrm{Supp}}
\DeclareMathOperator{\partsum}{\mathrm{sum}}

\newcommand{\lasttwo}{k_{\mathrm{last}}}
\newcommand{\lastlil}{\ell_{\mathrm{last}}}
\newcommand{\gr}[1]{\Gamma^{#1}}
\newcommand{\graphfxn}{\mathrm{graph}}
\newcommand{\matrixfxn}{\mathrm{matrix}}
\newcommand{\graphpart}{\mathrm{Part}}
\newcommand{\lgr}[2]{\Gamma^{{#1},{#2}}}

\newcommand{\heia}[2]{\mathrm{ht}_{\gr{#1}}({#2})}
\newcommand{\lheia}[3]{\mathrm{ht}_{\lgr{#1}{#2}}({#3})}
\newcommand{\ordk}[1]{\mathrm{n}_{#1}}
\newcommand{\ordg}[1]{\mathrm{n}_{\gr{#1}}}
\newcommand{\lordg}[2]{\mathrm{n}_{\lgr{#1}{#2}}}
\newcommand{\LO}[1]{\mathsf{T}^{#1}}
\newcommand{\SH}[1]{\mathsf{S}^{#1}}

\newcommand{\thief}[1]{\mathit{s}^{#1}}
\newcommand{\thiefh}[2]{\heia{#1}{\thief{#2}}}
\newcommand{\lthiefh}[3]{\lheia{#1}{#2}{\thief{#3}}}
\newcommand{\vic}[1]{\mathit{c}^{#1}}
\newcommand{\lvic}[2]{\mathit{c}^{{#1},{#2}}}
\newcommand{\vich}[2]{\heia{#1}{\vic{#2}}}
\newcommand{\lvich}[4]{\lheia{#1}{#2}{\lvic{#3}{#4}}}
 \newcommand{\term}{\mathit{end}}
\newcommand{\la}[1]{\lambda_{#1}}
\newcommand{\m}[2]{\mu^{#1}_{#2}}
\newcommand{\Index}[1]{\mathrm{Dom}(\gr{#1})}
\newcommand{\lIndex}[2]{\mathrm{Dom}(\lgr{#1}{#2})}

\newcommand{\flo}[1]{\lfloor{#1}\rfloor}
\newcommand{\cei}[1]{\lceil{#1}\rceil}

\newcommand{\pro}[2]{P_{#1}({#2})}
\newcommand{\lro}[2]{P_{#1}({#2})}
\newcommand{\welldef}{1}
\newcommand{\modify}{2}
\newcommand{\multisetsdecrease}{3}
\newcommand{\propdw}{4}
\newcommand{\converge}{5}
\newcommand{\conditionx}{6}
\newcommand{\pointers}{7}
\newcommand{\ordinalgood}{8}
\newcommand{\heightgood}{9}
\newcommand{\tallimpliesshort}{10}

\newcommand{\lpro}[3]{Q_{#1}^{#2}({#3})}
\newcommand{\lwelldef}{1}
\newcommand{\lmodify}{2}
\newcommand{\sciongrows}{3}
\newcommand{\lpropdw}{4}
\newcommand{\lconverge}{5}
\newcommand{\lconditionx}{6}
\newcommand{\lpointers}{7}
\newcommand{\lordinalgood}{8}

\newcommand{\dzz}{\frac{dz}{z}}
\newcommand{\duu}{\frac{du}{u}}

\newcommand{\sdfrac}[2]{\mbox{\small$\displaystyle\frac{#1}{#2}$}}

\newcommand{\sdzz}{\sdfrac{dz}{z}}
\newcommand{\sduu}{\sdfrac{du}{u}}

\title[Explicit constructions of connections on the projective line]{Explicit constructions of connections on the projective line with a maximally ramified irregular singularity}

\author[N. Livesay]{Neal Livesay}
\address{Roux Institute, Northeastern University, Boston, MA}
\email{n.livesay@northeastern.edu}
\author[D.~S. Sage]{Daniel S. Sage}
\address{Department of Mathematics, Louisiana State University, Baton Rouge, LA.}
\email{sage@math.lsu.edu}
\author[B. Nguyen]{Bach Nguyen}
\address{Department of Mathematics, Xavier University of Louisiana, New Orleans, LA.}
\email{bnguye22@xula.edu}

  

\begin{document}

\begin{abstract}
The Deligne--Simpson problem is an existence problem for connections with specified local behavior. Almost all previous work on this problem has restricted attention to connections with regular or unramified singularities.  Recently, the authors, together with Kulkarni and Matherne, formulated a version of the Deligne--Simpson problem where certain ramified singular points are allowed and solved it for the case of Coxeter connections, i.e., connections on the Riemann sphere with a maximally ramified singularity at zero and (possibly) an additional regular singular point at infinity. A certain matrix completion problem, which we call the Upper Nilpotent Completion Problem, plays a key role in our solution.  This problem was solved by Krupnik and Leibman, but their work does not provide a practical way of constructing explicit matrix completions. Accordingly, our previous work does not give explicit Coxeter connections with specified singularities.  In this paper, we provide a numerically stable and highly efficient algorithm for producing upper nilpotent completions of certain matrices that arise in the theory of Coxeter connections. Moreover, we show how the matrices generated by this algorithm can be used to provide explicit constructions of Coxeter connections with arbitrary unipotent monodromy in each case that such a connection exists.
\end{abstract}

\maketitle

\section{Introduction}
A fundamental question in the theory of meromorphic systems of linear differential equations on the Riemann sphere $\pp$ is whether there exists a differential equation with specified local behavior at its singularities. Consider the ODE $Y'+A(z)Y=0$, where $A(z)$ is an $n\times n$ matrix whose coefficients are rational functions over $\C$. This ODE determines a collection of local differential equations by expanding $A(z)$ as a Laurent series at each point in $\pp$. We view these local differential equations as ODEs with coefficients in the field of formal Laurent series in one variable. The local behavior of a meromorphic ODE is the collection of isomorphism classes of formal ODEs at the singular points.  One can then pose the natural question: given a finite subset of singular points in $\pp$ and a set of corresponding isomorphism classes of formal ODEs, does there exist a meromorphic ODE with this local behavior? The \emph{Deligne--Simpson problem} is a closely-related problem~\cite{sage2022meromorphic}.

Almost all previous work on the Deligne--Simpson problem has assumed that each singular point is either regular singular~\cite{CB03} (i.e., a simple pole) or ``unramified''~\cite{HirYam,Hiroe}.\footnote{A formal ODE is called \emph{unramified} if its Levelt--Turrittin form (an ODE version of Jordan canonical form) can be obtained without introducing roots of the local parameter.}  Recently, the authors, together with Kulkarni and Matherne, formulated a version of the Deligne--Simpson problem where "toral" ramified singular points are allowed and solved it for the case of  ODEs on the Riemann sphere with a maximally ramified singularity at zero and (possibly) an additional regular singular point at infinity~\cite{KLMNS}. This class of differential equations, which we call \emph{Coxeter ODEs}, includes important hypergeometric differential equations such as the Airy equation and the Kloosterman (or Frenkel--Gross) equation~\cite{KatzKloosterman,KatzExponential,FrGr}.  Explicitly, if we set  $\omega=\sum_{i=1}^{n-1}e_{i,i+1}+ze_{n,1}$---i.e., $\omega$ is the matrix with ones on the superdiagonal, $z$ in the lower-left entry, and zeroes elsewhere---then the Airy equation is  $Y'+z^{-2}\omega^{-1}Y=0$ and the Kloosterman equation is $Y'+z^{-1}\omega^{-1}Y=0$.

In this special case of the Deligne--Simpson problem, the specified local behavior at $0$ and $\infty$ is given by a polynomial $p(\omega^{-1})$ of degree $r$ with $\gcd(r,n)=1$ and an adjoint orbit (or simply a Jordan canonical form). The polynomial $p$ determines a ``formal type of slope $r/n$'' for the irregular singular point at $0$, while the adjoint orbit corresponds to the residue of the regular singular point at $\infty$.  For example, for the Kloosterman equation, $p(\omega^{-1})=\omega^{-1}$, $r=1$, and the Jordan canonical form is a single nilpotent Jordan block.  For the Airy equation, $p(\omega^{-1})=(\omega^{-1})^{n+1}$, $r=n+1$, and the adjoint orbit at $\infty$ is the zero orbit (meaning that $\infty$ is not a singular point).

A key discovery in \cite{KLMNS} is a relationship between the Deligne--Simpson problem for Coxeter ODEs and the following matrix completion problem, which we call the \emph{Upper Nilpotent Completion Problem}: 

\begin{quote}
\emph{Let $A$ be a nilpotent $n\times n$ matrix, let $\mu$ be the partition of $n$ with parts corresponding to the Jordan block sizes of $A$, and let $\la{}$ be a partition of $n$ dominating $\mu$. Show that there exists a strictly upper triangular matrix $X$ such that $A+X$ is nilpotent with Jordan block sizes $\la{}$.}
\end{quote}

To describe how this works, let us restrict to the simplest case where $0<r<n$, $p(\omega^{-1})=\omega^{-r}$, and the adjoint orbit is nilpotent with Jordan form corresponding to a partition $\lambda$ of $n$.  In this case, we showed in \cite{KLMNS} that there exists a connection with the specified local behavior if and only if $\lambda$ has at most $r$ parts, or equivalently if and only if the Jordan canonical form has at most $r$ blocks. For example, suppose that $A=N_r$, the matrix with ones in each entry of the $r$th diagonal and zeroes in all other entries. A (constructive) solution to the Upper Nilpotent Completion Problem with $A=N_r$ determines an explicit ODE with the given local behavior. Indeed, if $X$ is strictly upper triangular such that $N_r+X$ is nilpotent of type $\lambda$, then $Y'+z^{-1}(N_{n-r}^\top z^{-1} + N_r +X)Y=0$ has the desired properties.

The Upper Nilpotent Completion Problem was originally stated by Rodman and Shalom~\cite{RODMAN1992221}, and it was solved by Krupnik and Leibman~\cite{KrLe}.  This result, together with the solution of an extension of this problem to more general orbits by Krupnik~\cite{krupnik97}, plays a crucial role in \cite{KLMNS}. However, although Krupnik and Leibman describe an algorithm for constructing the desired nilpotent completion, the algorithm cannot be carried out effectively in practice to the best of our knowledge. Indeed, even in simple cases, the algorithm requires repeatedly transforming matrices into a form similar to Jordan canonical form, and we are not aware of any numerically stable way of carrying out their algorithm.  

In this paper, we give a numerically stable and highly efficient algorithm which constructs explicit (as well as binary and sparse) solutions to the Upper Nilpotent Completion Problem for the special case that $A=N_r$.  If $\lambda$ is any partition of $n$ with at most $r$ parts, the output of this algorithm is a strictly upper triangular binary matrix $X_\lambda$ such that $N_r+X_\lambda$ is nilpotent of type $\lambda$. As a result, the ODEs $Y'+z^{-1}(N_{n-r}^\top z^{-1} + N_r +X_\lambda)Y=0$ are explicit ``homogeneous'' Coxeter ODEs with a maximally ramified irregular singularity of slope $r/n$ at $0$ and unipotent monodromy of type $\lambda$ at $\infty$.

The Deligne--Simpson problem is often stated in the equivalent language of connections on trivializable vector bundles over $\pp$. The meromorphic ODE $Y'+A(z)Y=0$ with $A(z)$ an $n\times n$ matrix corresponds to the meromorphic connection $d+A(z)dz$ on a rank $n$ trivial vector bundle.  In the remainder of this paper, we will consider connections instead of ODEs.

\subsection*{Acknowledgements} The authors received support from the SQuaRE program of the American Institute of Mathematics and are grateful to AIM for its hospitality. The authors would like to thank the American Institute of Mathematics for its hospitality during an AIM SQuaRE where much of the work for this paper was completed.  N.L. received support from the Roux Institute and the Harold Alfond Foundation. D.S.S. received support from Simons Collaboration Grant 637367. B.N. received support from an AMS--Simons Travel Grant.

\section{Nilpotent orbits and the dominance order}

\subsection{Integer partitions}\label{partsection} We define a \emph{partition of a positive integer $n$} to be a multiset of positive integers that sum to $n$. Let $\partmult_{\la{}}(x)$ denote the multiplicity in a partition $\la{}$ of an integer $x$. We allow the multiplicity to take the value zero. Define the \emph{support $\partsupp(\la{})$ of a partition $\la{}$} to be the set of integers with positive multiplicity in $\la{}$; i.e., $\partsupp(\la{})=\{x\in\Z_{> 0}: \partmult_{\la{}}(x)>0 \}$. If $\partsupp(\la{})=\{\la{1},\la{2},\ldots,\la{k}\}$, then we represent the partition $\la{}$ as $\{\la{1}^{\partmult_{\la{}}(\la{1})},\la{2}^{\partmult_{\la{}}(\la{2})},\ldots,\la{k}^{\partmult_{\la{}}(\la{k})}\}$. Define $|\la{}|$ to be the sum $\sum_{x\in\partsupp(\la{})}\partmult_{\la{}}(x)$. Thus, if $\la{}$ is a partition of $n$, then $\sum_{x\in\partsupp(\la{})}\left(\sum_{i=1}^{\partmult_{\la{}}(x)}x\right)=n$; each of the $|\la{}|$ summands in the expansion of this double summation is called \emph{a part in $\la{}$}.

We find it convenient to sometimes view a partition $\la{}$ as a monotonically decreasing tuple $(\la{i})_{i=1}^{|\la{}|}$ of the parts in $\la{}$. The \emph{dominance order} (also known as the \emph{majorization order}, e.g., \cite{KrLe}) on the set of partitions of $n$ is the partial order $\succeq$ defined by $\la{}\succeq\m{}{}$ if and only if $|\la{}|\le |\m{}{}|$ and $\sum_{i=1}^j\la{i}\ge\sum_{i=1}^j\m{}{i}$ for all $j\in[1,|\la{}|]$.

\subsection{Nilpotent orbits}
If $R$ is a commutative ring with unity, then the general linear group $\GL_n(R)$ of degree $n$ over $R$ is the group of $n\times n$ matrices over $R$ with invertible determinant.  Its Lie algebra $\fgl_n(R)=\Lie(\GL_n(R))$ is the vector space of $n\times n$ matrices over $R$, equipped with the Lie bracket $[\cdot,\cdot]$ defined by $[X,Y]=XY-YX$.  We also view elements of $\fgl_n(R)$ as endomorphisms of $R^n$. The adjoint action of $\GL_n(R)$ on $\fgl_n(R)$ is defined by $\Ad_g(X)=gXg^{-1}$, for any $g\in\GL_n(R)$ and $X\in\fgl_n(R)$. In this paper, the term ``adjoint orbit'' always refers to a complex adjoint orbit. We denote the adjoint orbit of an element $X\in\fgl_n(\C)$ by $\orb_X$; i.e., $\orb_X = \{\Ad_g(X) \mid g\in\GL_n(\C)\}$.

An element $X\in\fgl_n(\C)$ is \emph{nilpotent} if $X^m=0$ for some $m>0$. The adjoint orbit $\orb_X$ of a nilpotent element $X\in\fgl_n(\C)$ is called a \emph{nilpotent orbit}. Any nilpotent orbit $\orb$ contains a block-diagonal representative in Jordan canonical form, which is unique up to permutations of its blocks (known as ``Jordan blocks''). Hence, there is a bijective correspondence between the set of nilpotent orbits in $\fgl_n(\C)$ and the set of partitions of $n$, where each nilpotent orbit corresponds with the multiset of its Jordan block sizes. We say that $\orb$---and each element of $\orb$---\emph{has type $\la{}$} if the sizes of the Jordan blocks for $\orb$ are given by the partition $\la{}$.

The set of nilpotent orbits are partially ordered: $\orb^1$ is less than or equal to $\orb^2$ if and only if $\orb^1$ is contained in the Zariski closure of $\orb^2$. This partial order is known as the \emph{closure order}. The bijection between nilpotent orbits and partitions described above defines a poset isomorphism when the sets are endowed with the closure order and the dominance order respectively~\cite{cm93}.

Next, we define two nilpotent matrices that arise in the study of Coxeter connections (see Section~\ref{sec:coxeter}).

\begin{definition}\label{def:Nr}Let $0<r<n$. Define $N_r\in\fgl_n(\C)$ (resp. $E_r\in\fgl_n(\C)$) to be the matrix with ones in each entry of the $r$th subdiagonal (resp. $(n-r)$th superdiagonal) and zeroes in all other entries.\end{definition}

\begin{figure}[ht!]
  \[N_1=\begin{bmatrix}0&0&0\\1&0&0\\0&1&0\end{bmatrix},\qquad E_1=\begin{bmatrix}0&0&1\\0&0&0\\0&0&0\end{bmatrix},\quad N_2=\begin{bmatrix}0&0&0\\0&0&0\\1&0&0\end{bmatrix},\quad 
  E_2=\begin{bmatrix}0&1&0\\0&0&1\\0&0&0\end{bmatrix}\]
\caption{The matrices $N_1$, $E_1$, $N_2$, and $E_2$ in $\fgl_3(\C)$.}
\end{figure} 

The following lemma shows that the orbit $\orb_{N_r}$ of $N_r$ is the minimal orbit with $r$ blocks and its closure consists of all orbits with at most $r$ blocks.  The proof is given in \cite[\S 2]{KLMNS}.

\begin{lem}\label{lem:Nr}
Suppose $0<r<n$.  Let $\rmd$ be the remainder when dividing $n$ by $r$.  Then:
\begin{enumerate}
\item $N_r$ is nilpotent and has type $\{\lceil n/r \rceil^{\rmd},\lfloor n/r \rfloor^{r-\rmd}\}$; and  
\item\label{lem:Nr:part2} a partition $\lambda$ of $n$ dominates $\{\lceil n/r \rceil^{\rmd},\lfloor n/r \rfloor^{r-\rmd}\}$ if and only if $|\lambda|\le r$.
\end{enumerate}
\end{lem}

\section{Coxeter connections and the Deligne--Simpson Problem}\label{sec:coxeter}

\subsection{Connections with maximally ramified singularities}
Let $V$ be a rank $n$ trivializable vector bundle over the Riemann sphere $\pp$ endowed with a meromorphic connection $\gc$. After fixing a trivialization, the connection has the form $d+M(z)\dzz$, where $M(z)\in\fgl_n(\C(z))$ is an $n\times n$ matrix of rational functions.  The local behavior at $y\in\pp$ is determined by the associated formal connection at $y$.  Explicitly, this is obtained by expanding $M(z)$ as a Laurent series in term of a local parameter $u$ at $y$: $z-y$ if $y$ is finite and $z^{-1}$ otherwise (i.e., if $y=\infty$).  Changing the trivialization of the global or formal vector bundle induces an action on the connection matrix called \emph{gauge change}. If $d+A(u)\dzz$ is a formal connection with $A(u)\in\fgl_n(\C(\!(u)\!))$, then  $g\in\GL_n(\C(\!(u)\!))$ acts on the connection operator via $g\cdot (d+A(u))\duu=d+(gA(u)g^{-1})\duu-(dg)g^{-1}$.  In the global case, $g\in\GL_n(\C)$, so the nonlinear gauge term is zero.

Let $y$ be a singular point of the connection, and denote the induced formal connection at $y$ by \[d+(M_{-r}u^{-r}+M_{-r+1}u^{-r+1}+\cdots)\sduu,\] where $M_i\in \fgl_n(\C)$ and $r\ge 0$. When the leading term $M_{-r}$ is well-behaved, it gives important information about the connection. For example, if $M_{-r}$ is not nilpotent, then the integer $r$ is an invariant of the connection known as the \emph{slope} at $y$. The slope can roughly be viewed as a measure of the irregularity of the singularity; a singular point $y$ is \emph{regular singular} if the slope is zero, and irregular otherwise. If $M_{-r}$ is regular semisimple (i.e., diagonalizable with distinct eigenvalues), then a classical result due to Wasow~\cite{Was} states that the connection is locally gauge equivalent to a connection of the form 
\[d+(D_{-r}u^{-r}+D_{-r+1}u^{-r+1}+\cdots+D_0)\sduu,\]
where $D_i\in\fgl_n(\C)$ are diagonal and $D_{-r}$ is regular semisimple. Diagonal representatives of this form are called \emph{regular unramified formal types}.\footnote{Any formal connection can be put into upper triangular form after passing to a finite extension of $\C(\!(u)\!)$. It is called \emph{unramified} if this can be done over the ground field and \emph{ramified} otherwise.}

However, there are many singularities for which the leading term of this expansion is nilpotent, regardless of the choice of formal trivialization.  Indeed, the slope at $y$ can be any nonnegative rational number with denominator at most $n$, and if this slope is not an integer, the leading term will always be nilpotent.  For example, the Frenkel--Gross connection~\cite{FrGr} \[d+(E_1z^{-1}+N_1)\sdzz\]  has two singular points, a regular singular point at $\infty$ and an irregular singular point at $0$ with slope $1/n$.  This is the smallest possible slope of an irregular singularity~\cite{KS1}.   The Frenkel--Gross connection is \emph{maximally ramified} at $y=0$, meaning that the slope has the largest possible denominator (i.e., $n$) when reduced to lowest terms.

To better understand formal connections with nonintegral slope, Bremer and Sage developed a generalization of leading terms known as \emph{fundamental strata}~\cite{BrSa1,BrSa2,BrSa3,BrSa5}. Every formal connection  ``contains'' a fundamental stratum \cite[Lemma 4.5]{BrSa1}, and the slope of a connection is equal to the ``depth'' of any fundamental stratum it contains (\cite[Theorem 4.10]{BrSa1},\cite[Theorem 2.14]{BrSa3}). Regular semisimple leading terms are generalized by \emph{regular strata}, and a generalization of Wasow's result states that any connection containing a regular stratum can be ``diagonalized'' into a \emph{ramified formal type}~\cite{BrSa1,BrSa5}. In particular, if a singularity has slope equal to $r/n$ for some $r$ relatively prime with $n$---i.e., if the singularity is maximally ramified---then the connection is locally gauge equivalent to a connection of the form $d+p(\omega^{-1})\dzz$, where $p$ is a polynomial of degree $r$ and $\omega^{-1}=E_1z^{-1}+N_1$ \cite[Theorem 4.4]{KLMNS}. The one-forms $p(\omega^{-1})\dzz$ are the ``maximally ramified formal types of slope $r/n$''.

\subsection{Moduli spaces and Coxeter connections}
An important problem in the study of meromorphic connections is the extent to which  a globally defined connection is determined by its local behavior. Here, ``local behavior'' consists of:
\begin{itemize}
\item a nonempty, finite set of irregular singular points $\{x_i\}_i$;
\item a corresponding collection $\bfA=(\ftype_i)_i$ of formal types\footnote{A formal type may be viewed as a rational canonical form for a formal connection. See \cite{sage2022meromorphic} for more details.} $\ftype_i$;
\item  a finite set of regular singularities $\{y_j\}_j$ disjoint from $\{x_i\}_i$; and
\item a corresponding collection $\bfO=(\orb_j)_j$ of ``nonresonant''\footnote{An adjoint orbit is called \emph{nonresonant} if no two eigenvalues differ by a nonzero integer.}  adjoint orbits $\orb_j$.
\end{itemize}

We consider the category $\Cat(\bfA,\bfO)$ of connections $\gc$ defined over the rank $n$, trivializable vector bundle $V$ on $\pp$ satisfying the following properties:
\begin{itemize}
\item $\gc$ has an irregular singularity at each $x_i$, a regular singularity at each $y_j$, and no other singularities;
\item at each $x_i$, $\gc$ is ``framable'' (see, e.g., \cite{Sa17}) and has formal type $A_i$; and
\item at each $y_j$, $\gc$ has residue in $\orb_j$.
\end{itemize}

Boalch gave a general construction of the corresponding moduli space $\M(\bfA,\bfO)$ in the case that each of the formal types are diagonal with regular semisimple leading term~\cite[Proposition 2.1]{Boa}. This construction was extended to include ``toral'' ramified formal types by Bremer and Sage~\cite[Theorem 5.6]{BrSa1}. 

Here, we give a relatively simple version of this construction for the important special case of connections on $\pp$ with a maximally ramified singularity at $0$, (possibly) a regular singularity at $\infty$, and no other singularities~\cite[Proposition 5.3]{KLMNS}. Such connections are called \emph{Coxeter connections} (for reasons discussed in~\cite{KS2,KLMNS}); they include both the Frenkel--Gross and Airy connections.

Let $B\subseteq\GL_n(\C)$ be the subgroup of upper triangular matrices.  Its Lie algebra $\fb$ is the set of upper triangular matrices in $\fgl_n(\C)$. Let $I$ denote the ``Iwahori subgroup'' of $\GL_n(\C[\![z]\!])$ corresponding to $B$, and let $\iwa$ denote its Lie algebra. Explicitly, this means that $I$ (resp. $\iwa$) is the preimage of $B$ (resp. $\fb$) via the map $\GL_n(\C[\![z]\!])\rightarrow\GL_n(\C)$ (resp. $\fgl_n(\C[\![z]\!])\rightarrow\fgl_n(\C)$) induced by the ``evaluation at zero'' map $z\mapsto 0$. 

Recall that the residue of a one-form $(\sum_{i=-r}^\infty M_iz^i)\dzz$ with $r\ge 0$ is $M_0$.

\begin{prop}[Proposition 5.3 in \cite{KLMNS}]
\label{thm:modspace}
Let $\ftype$ be a maximally ramified formal type, and let $\orb$ be a nonresonant adjoint orbit. Then
\[\M(\ftype,\orb) \cong
  \{(\alpha,Y) \mid \alpha\in\fgl_n(\C[z^{-1}])\dzz, Y\in\orb,
  \alpha+\iwa\subseteq\Ad(I)(\ftype)+\iwa,\text{ and
  }\Res(\alpha)+Y=0\}/B .\]
\end{prop}

\subsection{The Deligne--Simpson problem}
Very little is known about these moduli spaces when ramified singular points are allowed.  Indeed, it is not even known when these moduli spaces are nonempty; this is the \emph{Deligne--Simpson problem}.  The original Deligne--Simpson problem, involving connections with only regular singular points, was solved by Crawley--Boevey in 2003~\cite{CB03}. The unramified version, where unramified singular points are allowed, was solved more recently by Hiroe~\cite{Hiroe}. Recently, we solved the Deligne--Simpson problem for  Coxeter connections~\cite[Theorem 5.4]{KLMNS} together with Kulkarni and Matherne. Theorem~\ref{thm:ds} is a restatement of this result for the special case of \emph{homogeneous Coxeter connections}, i.e., Coxeter connections where the formal type $\ftype$ is a monomial. Without loss of generality, we can assume that $\ftype=\omega^{-r}\dzz$ (with $\gcd(r,n)=1$)~\cite{KS2}.

\begin{thm}\label{thm:ds}
Let $r$ and $n$ be relatively prime positive integers, and let $\orb_{\la{}}\subseteq\fgl_n(\C)$ be the nilpotent orbit of type $\la{}$. Then $\M(\omega^{-r}\dzz,\orb_{\la{}})$ is nonempty if and only if $|\la{}|\le r$.  In particular, this is always the case if $r>n$.
\end{thm}

While the results of \cite{KLMNS} determine exactly when  $\M(\omega^{-r}\dzz,\orb_{\la{}})$ is nonempty, they do not provide an explicit element of the moduli space.  When $r>n$, it is easy to find such an connection.  Indeed, if $X\in\orb_{\la{}}$ is strictly upper triangular (for example, if $X$ is the Jordan canonical form of $\orb_{\la{}}$), then  $d+(\omega^{-r}+X)\dzz$ has the desired local behavior.    Constructing an explicit connection in the moduli space is more complicated for $r<n$ and $|\la{}|\le r$, but it can be translated into a certain problem in matrix theory.

In linear algebra, an ``upper matrix completion problem'' is a problem where one studies various properties of the cosets $A+\fu$, where $A\in\fgl_n(\C)$ and $\fu$ is the space of strictly upper triangular matrices.  These problems have attracted great interest in the matrix theory community; see, for example, \cite{ball1990eigenvalues,gohberg1992bounds,gurvits1992controllability,woerdeman1989minimal,BeKr,KrRo,RODMAN1992221,KrLe,krupnik97}, as well as the survey in \cite[Chapter 35]{hogben}. Here, we are interested in the Upper Nilpotent Completion Problem: given $A$ nilpotent, determine the nilpotent orbits which intersect the coset $A+\fu$.  The most general result on this problem was conjectured by Rodman and Shalom~\cite{RODMAN1992221} and proved by Krupnik and Leibman~\cite{KrLe}.  It states that if $\mu$ is the partition of $n$ corresponding to the nilpotent orbit of $A$ and $\la{}$ is a partition dominating $\mu$, then there exists an element in $(A+\fu)\cap\orb_{\la{}}$.

Let us now restrict to the case $A=N_r$, and let $\mu_r$ be the corresponding partition.  It is shown in \cite{KLMNS} that the partitions dominating $\mu_r$ are precisely those with at most $r$ parts.  Let $\lambda$ be such a partition.  If $X$ is a strictly upper triangular matrix such that $N_r+X$ is nilpotent of type $\lambda$, then one can show that $d+(\omega^{-r}+X)\dzz$ corresponds to the element $B\cdot(\alpha,Y)\in\M(\omega^{-r}\dzz,\orb_{\la{}})$ with $\alpha=(\omega^{-r}+X)\dzz$ and $Y=-(N_r+X)$.  Krupnik and Leibman's result guarantees the existence of such an $X$.  However, their proof, while constructive in theory, does not seem possible to carry out in practice.   Indeed, their algorithm requires repeatedly transforming matrices into special Jordan-like forms, and even when starting with $A=N_r$, we are not aware of a numerically stable way of carrying out this algorithm.

In the next section, we specify a numerically stable and highly efficient algorithm that generates explicit solutions to the Upper Nilpotent Completion Problem for $N_r$, and hence leads to the construction of explicit Coxeter connections with arbitrary specified local behavior.

\section{The algorithm}\label{sec:algo}
In this section, we introduce our algorithm. The specification of the algorithm (in Section~\ref{sec:algospec}) and its proof of correctness (in Section~\ref{sec:correctness}) are based on many of the same graph-theoretic tools used by Krupnik and Leibman~\cite{KrLe}. Section~\ref{sec:graph} mostly recalls terminology and a key result (Proposition~\ref{thm:downward}) from \cite{KrLe}. Section~\ref{sec:transform} recalls the notion of a ``graph transformation'' from \cite{KrLe}, and introduces a class of graph transformations that we call ``graft transformations'', which are the fundamental operations in our algorithm.

\subsection{The graph of a matrix}\label{sec:graph}
Let $n$ be a positive integer, and let $V$ be a set of $n$ elements, called \emph{vertices}. Fix a bijection $\ordk{}\colon V\rightarrow[1,n]$, called an \emph{ordinal function}, which imposes a total ordering on $V$. We define a \emph{$\fgl_n$-graph} to be a pair consisting of:
\begin{enumerate}
\item a subset $E\subseteq V\times V$, whose elements are called \emph{arrows}; and
\item a \emph{weight function} $\alpha \colon E\rightarrow\Cs$.
\end{enumerate}
We define a bijective correspondence $\graphfxn$ between $\fgl_n(\C)$ and the set of $\fgl_n$-graphs as follows. Let $A=(\alpha_{i,j})_{i,j=1}^n$ be an element of $\fgl_n(\C)$. Define $\graphfxn(A)$ to be the $\fgl_n$-graph with the property that two vertices $u$ and $v$ are joined by an arrow $u\mapsto v$ if and only if $\alpha_{\ordk{}(v),\ordk{}(u)}\neq 0$, and the weight of each arrow $u\mapsto v$ is $\alpha_{\ordk{}(v),\ordk{}(u)}$. Define $\matrixfxn$ to be the inverse of the function $\graphfxn$.

\begin{rmk}
Starting in Section~\ref{sec:algospec}, we work exclusively with \emph{binary matrices}, i.e., matrices where each entry is either zero or one. Each arrow in the $\fgl_n$-graph of a binary matrix is weighted by one. Since such $\fgl_n$-graphs will be our focus, we assume (for the sake of simplicity) that any arrow $u\mapsto v$ lacking an explicit weight is implicitly weighted by one.
\end{rmk}

We consider embeddings $\phi$ of $\fgl_n$-graphs---or more precisely, their vertex sets---into $\N\times\N$. 
The image $\phi(v)=(x,y)$ of a vertex $v\in V$ has \emph{position} $x=x(v)$ and \emph{level} $y=y(v)$. Henceforth, we use the term ``graph'' to refer to a $\fgl_n$-graph equipped with an embedding into $\N\times\N$. We frequently refer to a vertex in a graph $\gr{}$ via its coordinates, and write $\ordg{}(x,y)$ to denote the ordinal of $\phi^{-1}((x,y))$.

A vertex $v$ is \emph{above} (resp. \emph{below}) a vertex $w$ if $x(v)=x(w)$ and $y(v)=y(w)+1$ (resp. $y(v)=y(w)-1$).  An arrow $v\mapsto w$ \emph{goes down} if $y(v)>y(w)$, and \emph{goes down-right} if $y(v)>y(w)$ and $x(v)\le x(w)$.  A graph $\gr{}$ is \emph{downward} if each of its arrows go down.  A graph $\gr{}$ is \emph{properly downward} if:
\begin{itemize}
\item $\gr{}$ is downward;
\item for every vertex $v$ with $y(v)>1$, there is a vertex $w$ below $v$ and an arrow $v\mapsto w$; and
\item every arrow $v\mapsto w$ with $y(v)=y(w)+1$ goes down-right.
\end{itemize}
We define the \emph{domain} $\Index{}$ of a graph $\gr{}$ with vertex set $V$ by $\Index{}=\{x(v) \mid v\in V\}$. The \emph{column in a graph $\gr{}$ at position $i$}---or, more concisely, \emph{column $i$ in $\gr{}$}---is $\{v\in V \mid x(v)=i\}$.  If $C$ is a column in a graph, then the \emph{height of $C$} is $\max\{y(v) \mid v\in C\}$.  If $i\in\Index{}$, then we denote the height of column $i$ in $\gr{}$ by $\heia{}{i}$, and the multiset of heights in $\gr{}$ by $\graphpart(\gr{})$.

\begin{prop}[{\cite[Proposition 3.2]{KrLe}}]
\label{thm:downward}
  Let $\gr{}$ be a graph on $n$ vertices.  If $\gr{}$ is downward, then $\matrixfxn(\gr{})$ is nilpotent.  If $\gr{}$ is properly downward, then $\graphpart(\gr{})$ is a partition of $n$ and $\matrixfxn(\gr{})$ is nilpotent of type $\graphpart(\gr{})$.
\end{prop}

We say that a column $C$ in a properly downward graph $\gr{}$ is \emph{a downward path in $\gr{}$} if the following condition holds: for each vertex $v\in C$, there exists an arrow $v\mapsto w$ if and only if $w$ is below $v$, and there exists an arrow $w\mapsto v$ if and only if $w$ is above $v$.

Let $X$ be a subset of the vertex set of a graph $\gr{}$. We say that $X$ is \emph{ordered by type-writer traversal} if, for each pair $v, w$ of vertices in $X$, $\ordk{}(v)<\ordk{}(w)$ if and only if one of the following holds: $y(v)>y(w)$, or $y(v)=y(w)$ and $x(v)<x(w)$. Note that if any subset of vertices is removed from a set ordered by type-writer traversal, then the resulting set remains ordered by type-writer traversal.

\begin{lem}\label{lem:Nrembed}
  Let $0<r<n$.  There exists a unique graph $\gr{}$ such that $\matrixfxn(\gr{})=N_r$, $\gr{}$ is properly downward, the vertex set is ordered by type-writer traversal, and the following conditions are satisfied:
  \begin{enumerate}
  \item $\Index{}=[1,r]$;
  \item each column is a downward path;
  \item each column has height $\flo{n/r}$ or $\cei{n/r}$; and
  \item if $1\le i<j\le r$, then $\heia{}{i}\le\heia{}{j}$.
  \end{enumerate}
\end{lem}

\begin{figure}[ht!]
\begin{center}
\begin{tikzpicture}
 [inner sep=0.5mm, place/.style={circle, draw}]

\node  (8) at (1,0)  [place] {\ \tiny{8 }  };
\node  (5) at (1,1)  [place] {\ \tiny{5 }  };
\node  (2) at (1,2)  [place] {\ \tiny{2 }  };
\node  (9) at (2,0)  [place] {\ \tiny{9 }  };
\node (6) at (2,1)  [place] {\ \tiny{6 }  };
\node  (3) at (2,2)  [place] {\ \tiny{3 }  };
\node (10) at (3,0) [place] {\tiny{10 }};
\node (7) at (3,1) [place] {\ \tiny{7 } } ;
\node (4) at (3,2) [place] {\ \tiny{4 }  };
\node (1) at (3,3) [place] {\ \tiny{1 }  };

\draw [thick,->] (1)-- (4);
\draw [thick,->] (4)-- (7);
\draw [thick,->,] (7)-- (10);

\draw [thick,->] (2)-- (5);
\draw [thick,->] (5)-- (8);

\draw [thick,->] (3)-- (6);
\draw [thick,->] (6)-- (9);

\end{tikzpicture}
 \caption{The graph $\gr{}$ of $N_3\in\fgl_{10}(\C)$ as described in Lemma~\ref{lem:Nrembed}.}
\label{fig:graph}
\end{center}
\end{figure}

\subsection{Transformations of a graph}\label{sec:transform}
Let $\gr{}$ be a graph, and let $\phi$ be the associated embedding of the vertex set into $\N\times\N$. Any change of the coordinates of the vertices of $\gr{}$ is called a \emph{geometric transformation of $\gr{}$}.  In other words, a geometric transformation of $\gr{}$ is the replacement of $\phi$ with another embedding $\phi^\prime$.  Geometric transformations preserve the matrix associated to $\gr{}$.  More generally, a \emph{transformation of a graph $\gr{}$} describes any finite sequence of the following:
\begin{itemize}
\item an addition or deletion of an arrow $u\mapsto v$;
\item a change of the weight $\alpha_{\ordk{}(v),\ordk{}(u)}$ of an arrow $u\mapsto v$; or
\item a geometric transformation.
\end{itemize}

We focus on one class of graph transformations that we call ``graft transformations'', defined in Definition~\ref{def:graft}. Our algorithm (Algorithm~\ref{alg:main}) performs a finite sequence of these graft transformations.

\begin{definition}\label{def:graft}
A graft transformation takes as input:
\begin{itemize}
\item a properly downward graph $\gr{}$;
\item two indices $\thief{}$ and $\vic{}$ (which we call the \emph{scion} and \emph{stock}, respectively) in $\Index{}$ with the property that column $\vic{}$ is a downward path and $\thief{}<\vic{}$; and
\item an integer $m$ satisfying $0<m\le\vich{}{}$ (the number of vertices that will be grafted).
\end{itemize}
To \emph{graft $m$ vertices in $\gr{}$ from $\vic{}$ to $\thief{}$}, perform the following two transformations on $\gr{}$:
\begin{itemize}
\item[Step 1:]Add an arrow $(\vic{},\vich{}{}-m+1)\mapsto(\thief{},\thiefh{}{})$.
\item[Step 2:]``Translate'' the top $m$ vertices of column $\vic{}$ to the top of column $\thief{}$; i.e., for each $i\in[1,m]$, change the embedding of vertex $(\vic{},\vich{}{}-m+i)$ to $(\thief{},\thiefh{}{}+i)$.
\end{itemize}
\end{definition} 

\begin{example} Let $\gr{}$ be the graph of $N_{3}\in\fgl_{10}$ as described in Lemma~\ref{lem:Nrembed}. Figure~\ref{fig:grafting} shows Steps 1 and 2 involved in grafting $m=2$ vertices in $\gr{}$ from column $\vic{}=3$ to column $\thief{}=2$.

\newsavebox\first
\newsavebox\second
\newsavebox\third

\begin{lrbox}{\first}
  \begin{tikzpicture}
 [inner sep=0.5mm, place/.style={circle, draw}]

\node  (8) at (1,0)  [place] {\ \tiny{8 }  };
\node  (5) at (1,1)  [place] {\ \tiny{5 }  };
\node  (2) at (1,2)  [place] {\ \tiny{2 }  };
\node  (9) at (2,0)  [place] {\ \tiny{9 }  };
\node (6) at (2,1)  [place] {\ \tiny{6 }  };
\node  (3) at (2,2)  [place] {\ \tiny{3 }  };
\node (10) at (3,0) [place] {\tiny{10 }};
\node (7) at (3,1) [place] {\ \tiny{7 } } ;
\node (4) at (3,2) [place] {\ \tiny{4 }  };
\node (1) at (3,3) [place] {\ \tiny{1 }  };

\draw [thick,->] (1)-- (4);
\draw [thick,->] (4)-- (7);
\draw [thick,->,] (7)-- (10);

\draw [thick,->] (2)-- (5);
\draw [thick,->] (5)-- (8);

\draw [thick,->] (3)-- (6);
\draw [thick,->] (6)-- (9);

  \end{tikzpicture}
\end{lrbox}

\begin{lrbox}{\second}
  \begin{tikzpicture}
 [inner sep=0.5mm, place/.style={circle, draw}]
\node  (8') at (7,0)  [place] {\ \tiny{8 }  };
\node  (5') at (7,1)  [place] {\ \tiny{5 }  };
\node  (2') at (7,2)  [place] {\ \tiny{2 }  };
\node  (9') at (8,0)  [place] {\ \tiny{9 }  };
\node (6') at (8,1)  [place] {\ \tiny{6 }  };
\node  (3') at (8,2)  [place] {\ \tiny{3 }  };
\node (10') at (9,0) [place] {\tiny{10 }};
\node (7') at (9,1) [place] {\ \tiny{7 } } ;
\node (4') at (9,2) [place] {\ \tiny{4 }  };
\node (1') at (9,3) [place] {\ \tiny{1 }  };

\draw [thick,->,] (1')-- (4');
\draw [thick,->,] (4')-- (7');
\draw [thick,->,] (7')-- (10');

\draw [thick,->] (2')-- (5');
\draw [thick,->] (5')-- (8');

\draw [thick,->] (3')-- (6');
\draw [thick,->] (6')-- (9');

\draw [thick,->,orange] (4')-- (3');
\end{tikzpicture}
\end{lrbox}

\begin{lrbox}{\third}
  \begin{tikzpicture}
 [inner sep=0.5mm, place/.style={circle, draw}]
\node  (8'') at (13,0)  [place] {\ \tiny{8 } };
\node  (5'') at (13,1)  [place] {\ \tiny{5 }  };
\node  (2'') at (13,2)  [place] {\ \tiny{2 }  };
\node  (9'') at (14,0)  [place] {\ \tiny{9 }  };
\node (6'') at (14,1)  [place] {\ \tiny{6 }  };
\node  (3'') at (14,2)  [place] {\ \tiny{3 }  };
\node (10'') at (15,0) [place] {\tiny{10 }};
\node (7'') at (15,1) [place] {\ \tiny{7 } } ;
\node (4'') at (14,3) [place] {\ \tiny{4 }  };
\node (1'') at (14,4) [place] {\ \tiny{1 }  };

\draw [thick,->,] (1'')-- (4'');
\draw [thick,->,] (4'')-- (7'');
\draw [thick,->,] (7'')-- (10'');

\draw [thick,->] (2'')-- (5'');
\draw [thick,->] (5'')-- (8'');

\draw [thick,->] (3'')-- (6'');
\draw [thick,->] (6'')-- (9'');

\draw [thick,->] (4'')-- (3'');

\end{tikzpicture}
\end{lrbox}

\begin{figure}[ht!]
\begin{tikzpicture}[red]
  \node (1) at(0,0) {
        \usebox\first
  };
  \node (2) at (6,0){\usebox\second};
  \node (3) at (12,0) {\usebox\third};
\draw[->, >= to] (2,0) to [edge label= Step 1] (4,0);  
\draw[->, >= to] (8,0) to [edge label=Step 2] (10,0);
\end{tikzpicture}
 \caption{On the left is the graph $\gr{}$ of $N_3\in\fgl_{10}$ as described in Lemma~\ref{lem:Nrembed}. On the right is the result of grafting $m=2$ vertices in $\gr{}$ from column $\vic{}=3$ to column $\thief{}=2$ (see Definition~\ref{def:graft}).}
\label{fig:grafting}
\end{figure}
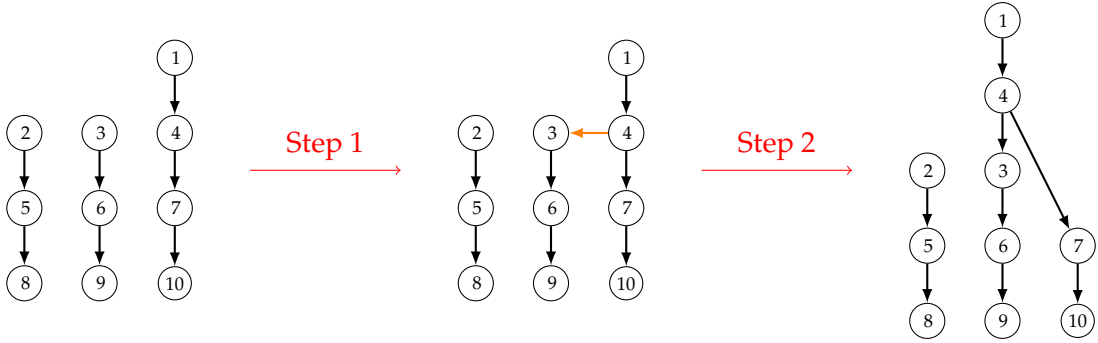
\end{example}

\begin{lem}\label{graftlemma}
  Let $\gr{1}$ be a properly downward graph, and let $\vic{}$, $\thief{}$ be indices in $\Index{1}$ with the property that column $\vic{}$ is a downward path and $\thief{}<\vic{}$.  Let $\gr{2}$ be the graph that is the result of grafting $m$ vertices from $\vic{}$ to $\thief{}$, where $0<m\le\vich{1}{}$. Then:
  \begin{enumerate}
  \item\label{graftlemmadomain} If $m=\vich{1}{}$, then $\Index{2}=\Index{1}-\{\vic{}\}$; otherwise $\Index{2}=\Index{1}$.
  \item\label{graftlemmaheight} $\heia{2}{\vic{}}=\heia{1}{\vic{}}-m$, $\heia{2}{\thief{}}=\heia{1}{\thief{}}+m$, and $\heia{2}{i}=\heia{1}{i}$ for all $i\in\Index{2}$ with $i\notin\{\thief{},\vic{}\}$.
  \item\label{graftlemmadownwardpath} If column $i$ is a downward path in $\gr{1}$ and $i\notin\{\thief{},\vic{}\}$, then column $i$ is a downward path in $\gr{2}$.
  \item\label{graftlemmaordinal} For each $i\in[1,m]$, $\ordg{2}(\thief{},\thiefh{1}{}+i)=\ordg{1}(\vic{},\vich{1}{}-m+i)$.  In particular, $\ordg{2}(\thief{},\thiefh{2}{})=\ordg{1}(\vic{},\vich{1}{})$. If $(x,y)$ is a vertex in $\gr{2}$ with $x\neq\thief{}$, then $(x,y)$ is a vertex in $\gr{1}$ and $\ordg{1}(x,y)=\ordg{2}(x,y)$.
  \item\label{graftlemmapropdw} If $\thiefh{1}{}>\vich{1}{}-m$, then $\gr{2}$ is properly downward.
  \item\label{graftlemmamodify} If $\ordg{1}(\thief{},\thiefh{1}{})<\ordg{1}(\vic{},\vich{1}{}-m+1)$, then $\matrixfxn(\gr{2})$ is obtained from $\matrixfxn(\gr{1})$ by changing a single upper triangular entry from zero to one.
  \end{enumerate}
\end{lem}

\subsection{The algorithm}\label{sec:algospec}
Our algorithm, Algorithm~\ref{alg:main}, is specified below. The algorithm takes as input two positive integers $r$ and $n$ with $r<n$, and a partition $\la{}$ of $n$ with at most $r$ parts. The algorithm consists of the following objects, each of which has a state that changes over the course of execution:
\begin{itemize}
\item a graph $\gr{}$;
\item two integers $\thief{}$ and $\vic{}$ (which we call ``column pointers''); and
\item two multisets $\LO{}$ and $\SH{}$.
\end{itemize}
The initial state of $\gr{}$ is the graph of $N_r\in\fgl_n(\C)$ as described in Lemma~\ref{lem:Nrembed}. The column pointers and multisets are initialized in lines \ref{alg:main:initstart}--\ref{alg:main:initend}. The algorithm makes use of three operations on multisets: $+$, $-$, and $\max$. To define these operations, let $\la{}$ and $\m{}{}$ be multisets. The sum $\la{}+\m{}{}$ is the multiset with multiplicity function $\partmult_{\la{}+\m{}{}}(x)=\partmult_{\la{}}(x)+\partmult_{\m{}{}}(x)$ for all $x$. The difference $\la{}-\m{}{}$ is the multiset with multiplicity function $\partmult_{\la{}-\m{}{}}(x)=\max\{0,\partmult_{\la{}}(x)-\partmult_{\m{}{}}(x)\}$ for all $x$. Finally, $\max(\la{})$ is the maximum value in $\partsupp(\la{})$.

After initialization, \emph{Loop 1} (lines \ref{alg:main:loop1start}--\ref{alg:main:loop1end}) and \emph{Loop 2} (lines \ref{alg:main:loop2start}--\ref{alg:main:loop2end}) are executed. We refer to these two loops as the \emph{primary loops} (as opposed to the ``Little Loop'' defined in lines \ref{alg:main:littleloopstart}--\ref{alg:main:littleloopend}). Each iteration of a primary loop performs at least one graft transformation to $\gr{}$, reduces the cardinality of $\SH{}$ and/or $\LO{}$ by one, and increases the column pointer $\vic{}$ (and possibly increases $\thief{}$ as well). In Section~\ref{sec:correctness}, we prove that Algorithm~\ref{alg:main} terminates after some finite number of iterations of the primary loops, and that the terminal state of $\gr{}$ satisfies the specified postconditions.

\begin{algorithm}
\caption{}
\label{alg:main}
\begin{algorithmic}[1]
\Require $r$ and $n$ are positive integers with $r<n$,
$\rmd$ is the remainder when $n$ is divided by $r$, \newline
$\gr{}$ is the graph of $N_r\in\fgl_n(\C)$ as described in Lemma~\ref{lem:Nrembed}, \newline $\la{}$ is a partition of $n$ with at most $r$ parts
\Ensure $\matrixfxn(\gr{})=N_r+X$ where $X$ is strictly upper triangular,
\newline $\matrixfxn(\gr{})$ is binary and nilpotent of type $\la{}$
\State\label{alg:main:initstart}$\LO{}\gets\{x^{\partmult_{\la{}}(x)} \mid x\in\partsupp(\la{}), x>\lceil n/r\rceil\}$
\If{$\rmd\neq 0$ {\bf and} $\partmult_{\la{}}(\cei{n/r})>\rmd$}
    \State $\LO{}\gets\LO{}+\{\cei{n/r}^{\partmult_{\la{}}(\cei{n/r})-\rmd}\}$
\EndIf
\State $\SH{}\gets\{x^{\partmult_{\la{}}(x)} \mid x\in\partsupp(\la{}), x<\flo{ n/r}\}$
\If{$\partmult_{\la{}}(\flo{n/r})>r-\rmd$}
    \State $\SH{}\gets\SH{}+\{\flo{n/r}^{\partmult_{\la{}}(\flo{n/r})-(r-\rmd)}\}$
\EndIf
\State $\thief{}\gets\min\{\partmult_{\la{}}(\flo{n/r}),r-\rmd\}+1$
\State\label{alg:main:initend}$\vic{}\gets\thief{}+1$
\While{$\SH{}$ is nonempty} \Comment{Loop 1} \label{alg:main:loop1start}
\If{$\max(\LO{})-\heia{}{\thief{}}=\heia{}{\vic{}}-\max(\SH{})+1$ {\bf and} $\heia{}{\vic{}}<\heia{}{\vic{}+1}$} \Comment{Case 1a}
	\State\label{alg:main:case1astart} graft $\heia{}{\vic{}+1}-\max(\SH{})$ vertices from column $\vic{}+1$ to column $\thief{}$
	\State $\SH{}\gets\SH{}-\{\max(\SH{})\}$
	\State $\LO{}\gets\LO{}-\{\max(\LO{})\}$
	\State $\thief{}\gets\vic{}$
	\State\label{alg:main:case1aend}$\vic{}\gets\vic{}+2$

\ElsIf{$\max(\LO{})-\thiefh{}{} > \vich{}{}-\max(\SH{})$} \Comment{Case 1b}
	\State graft $\vich{}{}-\max(\SH{})$ vertices from column $\vic{}$ to column $\thief{}$
	\State $\SH{}\gets\SH{}-\{\max(\SH{})\}$
	\State $\vic{}\gets\vic{}+1$
	
\ElsIf{$\max(\LO{})-\thiefh{}{} = \vich{}{}-\max(\SH{})$} \Comment{Case 1c}
	\State graft $\vich{}{}-\max(\SH{})$ vertices from column $\vic{}$ to column $\thief{}$
	\State $\SH{}\gets\SH{}-\{\max(\SH{})\}$
	\State $\LO{}\gets\LO{}-\{\max(\LO{})\}$
	\State $\thief{}\gets\vic{}+1$
	\State $\vic{}\gets\vic{}+2$
	
\ElsIf{$\max(\LO{})-\thiefh{}{} < \vich{}{}-\max(\SH{})$} \Comment{Case 1d}
	\State graft $\max(\LO{})-\thiefh{}{}$ vertices from column $\vic{}$ to column $\thief{}$
	\State $\LO{}\gets\LO{}-\{\max(\LO{})\}$
	\State $\thief{}\gets\vic{}$
	\State $\vic{}\gets\vic{}+1$ \label{alg:main:loop1end}
\EndIf
\EndWhile

	\algstore{some_name}
\end{algorithmic}
\end{algorithm}

\begin{algorithm}
\begin{algorithmic}[1]
    \algrestore{some_name}
\While{$\LO{}$ is nonempty} \Comment{Loop 2} \label{alg:main:loop2start}
	\While{$\max(\LO{})-\thiefh{}{}>\cei{n/r}$} \Comment{Little Loop} \label{alg:main:littleloopstart}
		\State graft $\vich{}{}$ vertices from column $\vic{}$ to column $\thief{}$ \label{alg:main:littleloopfirst}
		\State $\vic{}\gets\vic{}+1$ \label{alg:main:littleloopend}
	\EndWhile
	
	\If{$\max(\LO{})-\thiefh{}{}>\vich{}{}$ {\bf and} $\heia{}{\vic{}+1}=\cei{n/r}$} \Comment{Case 2a}
		\State graft $\heia{}{\vic{}+1}$ vertices from column $\vic{}+1$ to column $\thief{}$
		\State $\thief{}\gets\vic{}$
		\State $\vic{}\gets\vic{}+2$
	
	\ElsIf{$\max(\LO{})-\thiefh{}{}>\vich{}{}$ {\bf and} $\heia{}{\vic{}+1}=\flo{n/r}$} \Comment{Case 2b}
		\State graft $\vich{}{}$ vertices from column $\vic{}$ to column $\thief{}$
		\State $\vic{}\gets\vic{}+1$
		\State graft one vertex from column $\vic{}$ to column $\thief{}$
		\State $\thief{}\gets\vic{}$ 
		\State $\vic{}\gets\vic{}+1$
    \ElsIf{$\max(\LO{})-\thiefh{}{} = \vich{}{}$} \Comment{Case 2c}
		\State graft $\vich{}{}$ vertices from column $\vic{}$ to column $\thief{}$
		\State $\thief{}\gets\vic{}+1$
		\State $\vic{}\gets\vic{}+2$
    
	\ElsIf{$\max(\LO{})-\thiefh{}{} < \vich{}{}$} \Comment{Case 2d}
		\State graft $\max(\LO{})-\thiefh{}{}$ vertices from column $\vic{}$ to column $\thief{}$
		\State $\thief{}\gets\vic{}$
		\State $\vic{}\gets\vic{}+1$
	\EndIf
	\State $\LO{}\gets\LO{}-\{\max(\LO{})\}$ \label{alg:main:loop2end}
\EndWhile
\end{algorithmic}
\end{algorithm}

We have implemented Algorithm~\ref{alg:main} in the SageMath computer algebra system~\cite{sagemath}, as well as the NumPy array and matrix computing library~\cite{numpy} for the Python programming language.\footnote{Our Python implementation is publicly available at https:/\!/github.com/neallivesay/nilpotent-completions.} We have tested our SageMath implementation for all positive integers $r$ and $n$ and all partitions $\la{}$ of $n$ satisfying $0<r<n\le 30$ and $|\la{}|\le r$.

\newpage
\section{Proof of correctness}\label{sec:correctness}
In this section, we prove that Algorithm~\ref{alg:main} is correct. In other words, we prove that given any valid input, the algorithm is well-defined and terminates, and that the matrix associated to the terminal state of the graph $\gr{}$ satisfies the specified postconditions. With this goal in mind, let $r$ and $n$ be positive integers with $r<n$, let $\rmd$ be the remainder when $n$ is divided by $r$, and let $\la{}$ be a partition with at most $r$ parts. Initialize a graph $\gr{}$, multisets $\LO{}$ and $\SH{}$, and column pointers $\thief{}$ and $\vic{}$ as specified in lines \ref{alg:main:initstart}--\ref{alg:main:initend}.

Our proof is inductive on the number $k$ of times that a primary loop has executed; i.e., $k$ equals the sum of the number of times Loop 1 (lines \ref{alg:main:loop1start}--\ref{alg:main:loop1end}) has executed with the number of times Loop 2 (lines \ref{alg:main:loop2start}--\ref{alg:main:loop2end}) has executed. For each object (i.e., a graph, multiset, or pointer) $Z$, let $Z^k$ denote the state of $Z$ at the end of the $k$th iteration, with $Z^0$ denoting the state of $Z$ after initialization (i.e., after execution of lines \ref{alg:main:initstart}--\ref{alg:main:initend}). Fix once and for all a constant column pointer $\term = r-\min\{\rmd,\partmult_{\la{}}(\cei{n/r})\}$.

The proof involves a simultaneous induction over the following propositional functions of $k$:
  \begin{itemize}
  \item[$\pro{\welldef}{k}:$] Each instruction is well-defined in iteration $k$.
  \item[$\pro{\modify}{k}:$] $\matrixfxn(\gr{k})=\matrixfxn(\gr{k-1})+X$ for some strictly upper triangular matrix $X$.
  \item[$\pro{\multisetsdecrease}{k}:$] $\SH{k}\subseteq\SH{k-1}$ and $\LO{k}\subseteq\LO{k-1}$, with at least one subset relation strict.
  \item[$\pro{\propdw}{k}:$] $\gr{k}$ is properly downward.
  \item[$\pro{\converge}{k}:$] $\{\heia{k}{i}\mid i\in\Index{k}, i\notin [\vic{k},\term],i\neq\thief{k}\}=\la{}-(\LO{k}+\SH{k})$.
  \item[$\pro{\conditionx}{k}:$] If $i$ satisfies $\vic{k}\le i\le\term$, then:
    \begin{enumerate}
    \item $i\in\Index{k}$;
    \item column $i$ is a downward path;
    \item $\heia{k}{i}=\flo{n/r}$ or $\heia{k}{i}=\cei{n/r}$;
    \item if $i<j\le\term$, then $\heia{k}{i}\le\heia{k}{j}$; and
    \item if $\SH{k}$ is nonempty, then $\heia{k}{i}>\max(\SH{k})$, and if $\LO{k}$ is nonempty, then $\heia{k}{i}<\max(\LO{k})$.
    \end{enumerate}
    Moreover, the set of vertices $v$ with $\vic{k}\le x(v)\le\term$ are ordered by type-writer traversal.
  \item[$\pro{\pointers}{k}:$] $\thief{k}<\vic{k}$.  If $\LO{k}$ is nonempty, then $\thief{k},\vic{k}\in\Index{k}$ and $\vic{k}\le\term$.
  \item[$\pro{\ordinalgood}{k}:$] If $\SH{k}$ is nonempty, then $\ordg{k}(\thief{k},\thiefh{k}{k})<\ordg{k}(\vic{k},\max(\SH{k})+1)$.  If $\LO{k}$ is nonempty, then $\ordg{k}(\thief{k},\thiefh{k}{k})<\ordg{k}(\vic{k},\max\{1,\vich{k}{k}-(\max(\LO{k})-\thiefh{k}{k})+1\})$.
  \item[$\pro{\heightgood}{k}:$]If $\SH{k}$ is nonempty, then $\thiefh{k}{k}>\max(\SH{k})$.
  \item[$\pro{\tallimpliesshort}{k}:$] If $\SH{k}$ is nonempty, then $\LO{k}$ is nonempty.
  \end{itemize}
For all $i\in[1,\multisetsdecrease]$ (resp. for all $i\in[\propdw,\tallimpliesshort]$), the domain of $P_i$ is the set of all $k\ge 1$ (resp. all $k\ge 0$) such that $\SH{k-1}$ or $\LO{k-1}$ is nonempty. Lemma~\ref{lem:sufficiency} establishes the sufficiency of the universal quantifications of the above propositional functions for proving correctness of Algorithm~\ref{alg:main}.

\begin{lem}\label{lem:sufficiency}
If $P_i(0)$ holds for all $\propdw \le i \le \tallimpliesshort$, and $P_i(k)$ holds for all $1 \le i \le \tallimpliesshort$ and all $k\ge 1$ with the property that $\SH{k-1}$ or $\LO{k-1}$ is nonempty, then the algorithm terminates after some number of steps $\lasttwo$. Moreover, $\matrixfxn(\gr{\lasttwo})=N_r + X$ for some strictly upper triangular matrix $X$, and $\matrixfxn(\gr{\lasttwo})$ is nilpotent of type $\la{}$.
\end{lem}
  
\begin{proof} Suppose that the hypothesis of Lemma~\ref{lem:sufficiency} holds. Define $X_{\SH{}}=\{k\in\N \mid \SH{k-1}\text{ is nonempty}\}$ and $X_{\LO{}}=\{k\in\N \mid \LO{k-1}\text{ is nonempty}\}$. Proposition $\pro{\welldef}{k}$ holds, and thus the algorithm runs, for all iterations $k\in X_{\SH{}}\cup X_{\LO{}}$. Since $\pro{\tallimpliesshort}{k}$ holds for all $k\in X_{\SH{}}\cup X_{\LO{}}$, it follows that $X_{\SH{}}\subseteq X_{\LO{}}$. Since $\pro{\multisetsdecrease}{k}$ for all $k\in X_{\SH{}}\cup X_{\LO{}}$, it follows that $\max(X_{\SH{}})\le\max(X_{\LO{}})<\infty$; Loop 1 iterates for all $k\in X_{\SH{}}$, and Loop 2 iterates for all $k\in X_{\LO{}}-X_{\SH{}}$, with execution terminating after iteration $\lasttwo=\max(X_{\LO{}})$. Proposition $\pro{\propdw}{\lasttwo}$ implies that the terminal state, $\gr{\lasttwo}$, of the graph is properly downward. Thus $\matrixfxn(\gr{\lasttwo})$ is nilpotent of type $\la{}$ by Proposition~\ref{thm:downward} and $\pro{\converge}{\lasttwo}$. Since $\lro{\modify}{k}$ holds for all $k$ such that $\LO{k-1}$ is nonempty, it follows that $\matrixfxn(\gr{\lasttwo})=N_r+X$ for some strictly upper triangular matrix $X$.
 \end{proof}
 
The remainder of this section is dedicated to proving the hypothesis of Lemma~\ref{lem:sufficiency} via simultaneous induction. We begin by establishing a basis for induction. Propositions $\pro{\propdw}{0}$, $\pro{\conditionx}{0}$, $\pro{\pointers}{0}$, $\pro{\ordinalgood}{0}$, and $\pro{\heightgood}{0}$ follow trivially (mostly as a consequence of Lemma~\ref{lem:Nrembed}). Proposition $\pro{\converge}{0}$ holds since $\{\heia{0}{i} \mid i\in\Index{0}, i\notin [\vic{0},\term],i\neq\thief{0}\} = \{\heia{0}{i} \mid i\in[1,\thief{0})\cup (\term,r]\} =\{\flo{n/r}^{\min\{\partmult_{\la{}}(\flo{n/r}),r-\rmd\}} , \cei{n/r}^{\min\{\rmd,\partmult_{\la{}}(\cei{n/r})\}}\}=\la{}-(\LO{0}+\SH{0})$. To prove $\pro{\tallimpliesshort}{0}$, suppose for a contradiction that $\SH{0}$ is nonempty and $\LO{0}$ is empty. Then at most $r^\prime$ parts in $\la{}$ equal $\cei{n/r}$. The remaining parts are at most $\flo{n/r}$, with at least one part being strictly less. But this implies $n=\partsum(\lambda)<(r-r^\prime)\flo{n/r}+r^\prime\cei{n/r}=n$, a contradiction.

To prove the inductive step, let $k$ be a positive integer with the property that $\SH{k-1}$ is nonempty or $\LO{k-1}$ is nonempty.  Suppose that the algorithm executes for $k-1$ well-defined iterations of the primary loops, and suppose that $\pro{i}{j}$ holds for each $1\le j< k$ and each $1\le i\le 11$. We proceed to prove that $\pro{i}{k}$ is true for all $1\le i\le 11$. One of the following two cases is satisfied at the start of the $k$th iteration:
\begin{itemize}
\item[Case 1:]$\SH{k-1}$ is nonempty; or
\item[Case 2:]$\SH{k-1}$ is empty and $\LO{k-1}$ is nonempty.
\end{itemize}
Case 1 is considered in Section~\ref{sec:case1} and Case 2 is considered in Section~\ref{sec:case2}.

  \subsection{Suppose Case 1 is satisfied at the start of iteration $k$}\label{sec:case1} That is, suppose that $\SH{k-1}$ is nonempty. Then the $k$th iteration is an iteration of Loop 1. Consider the following four conditional expressions:
  \begin{itemize}
  \item[Case 1a:] $\max(\LO{k-1})-\heia{k-1}{\thief{k-1}}=\heia{k-1}{\vic{k-1}}-\max(\SH{k-1})+1$

    and $\heia{k-1}{\vic{k-1}}<\heia{k-1}{\vic{k-1}+1}$;
  \item[Case 1b:] $\max(\LO{k-1})-\thiefh{k-1}{k-1} > \vich{k-1}{k-1}-\max(\SH{k-1})$ and Case 1a is not satisfied;
  \item[Case 1c:] $\max(\LO{k-1})-\thiefh{k-1}{k-1} = \vich{k-1}{k-1}-\max(\SH{k-1})$; and
  \item[Case 1d:] $\max(\LO{k-1})-\thiefh{k-1}{k-1} < \vich{k-1}{k-1}-\max(\SH{k-1})$.
  \end{itemize}
  Note that $\pro{\tallimpliesshort}{k-1}$ implies that $\LO{k-1}$ is nonempty.  Hence each of the conditional expressions in Cases 1b, 1c, and 1d are well-defined.  To verify that the Case 1a conditional expression is well-defined, it suffices to show that $\vic{k-1}+1\le\term$ whenever
  \begin{equation}\label{uniquelabel} \max(\LO{k-1})-\heia{k-1}{\thief{k-1}}=\heia{k-1}{\vic{k-1}}-\max(\SH{k-1})+1.\end{equation}
  Suppose, for a contradiction, that $(\ref{uniquelabel})$ holds but $\vic{k-1}+1>\term$. 
  Then $\pro{\conditionx}{k-1}$ implies that $\vic{k-1}=\term$.  Hence
\[\begin{array}{rll}
    n &=\sum_{\{i \mid i\in\Index{k-1}, i\notin [\vic{k-1},\term],i\neq\thief{k-1}\}}\heia{k-1}{i}+\heia{k-1}{\thief{k-1}}+\heia{k-1}{\vic{k-1}} \\
      &=\partsum(\la{}-(\LO{k-1}+\SH{k-1}))+\heia{k-1}{\thief{k-1}}+\heia{k-1}{\vic{k-1}} & \text{(by $\pro{\converge}{k-1}$)}\\
      &<\partsum(\la{}-(\LO{k-1}+\SH{k-1}))+\max(\LO{k-1})+\max(\SH{k-1}) & \text{(by (\ref{uniquelabel}))}\\
      &\le n,
  \end{array}\]
a contradiction.  Hence $\vic{k-1}+1\le\term$ if (\ref{uniquelabel}) is satisfied, which implies that the conditional expression for Case 1a is well-defined.

It is easily verified that the conditional expressions for Cases 1a--d are mutually exclusive and exhaustive.  A different set of instructions executes for each of the four cases.  We now prove the inductive step for Case 1a; the proofs for Cases 1b--d are similar and thus are omitted.

\subsubsection{Suppose that Case 1a is satisfied at the start of iteration $k$.} As discussed above, it follows that $\vic{k-1}+1\le\term$. Since $\heia{k-1}{\vic{k-1}}<\heia{k-1}{\vic{k-1}+1}$, it follows (by $\pro{\conditionx}{k-1}$) that $\heia{k-1}{\vic{k-1}}=\flo{n/r}$, $\heia{k-1}{\vic{k-1}+1}=\cei{n/r}$, and $\flo{n/r}\neq\cei{n/r}$.

We walk through the execution of the Case 1a instructions (i.e., lines \ref{alg:main:case1astart}--\ref{alg:main:case1aend}). The graph $\gr{k}$ is formed by grafting $\heia{k-1}{\vic{k-1}+1}-\max(\SH{k-1})$ vertices in $\gr{k-1}$ from column $\vic{k-1}+1$ to column $\thief{k-1}$. This grafting operation is well-defined since $\gr{k-1}$ is properly downward and column $\vic{k-1}+1$ is a downward path in $\gr{k-1}$ (by $\pro{\conditionx}{k-1}$). Finally, multisets and column pointers are reassigned---resulting in $\SH{k}=\SH{k-1}-\{\max(\SH{k-1})\}$, $\LO{k}=\LO{k-1}-\{\max(\LO{k-1})\}$, $\thief{k}=\vic{k-1}$, and $\vic{k}=\vic{k-1}+2$---and iteration $k$ concludes. Since each of the expressions evaluated and instructions executed during iteration $k$ are well-defined, $\pro{\welldef}{k}$ follows. Proposition $\pro{\multisetsdecrease}{k}$ is clear.

  The remainder of the proof for Case 1a largely relies on Lemma~\ref{graftlemma}.  By $\pro{\heightgood}{k-1}$, it follows that $\thiefh{k-1}{k-1} > \vich{k-1}{k-1} - (\vich{k-1}{k-1}-\max(\SH{k-1}))$. Thus Lemma~\ref{graftlemma}(\ref{graftlemmapropdw}) implies $\pro{\propdw}{k}$. It is straight-forward to verify $\pro{\conditionx}{k}$. Proposition $\pro{\heightgood}{k}$ follows immediately since $\thief{k}=\vic{k-1}\in\Index{k}$ and $\thiefh{k}{k}=\flo{n/r}$.

  Proposition $\pro{\modify}{k}$ follows from the combination of Lemma~\ref{graftlemma}(\ref{graftlemmamodify}) and the fact that
  \[\begin{array}{rll}
      \ordg{k-1}(\thief{k-1},\thiefh{k-1}{k-1})&<\ordg{k-1}(\vic{k-1},\max(\SH{k-1})+1) & (\text{by }\pro{\ordinalgood}{k-1})\\
      &<\ordg{k-1}(\vic{k-1}+1,\max(\SH{k-1})+1) & (\text{by }\pro{\conditionx}{k-1})\\
      &=\ordg{k-1}(\vic{k-1},\vich{k-1}{k-1} \\
      &\hspace{1cm}- (\vich{k-1}{k-1} - \max(\SH{k-1}))).
    \end{array}\]
Proposition $\pro{\converge}{k}$ follows since
  \[\begin{array}{rll}
      &\{\heia{k}{i} \mid i\in\Index{k},i\notin[\vic{k},\term],i\neq\thief{k}\} \\
      &= \{\heia{k-1}{i} \mid i\in\Index{k-1},i\notin[\vic{k-1},\term],i\neq\thief{k-1}\}+\{\thiefh{k}{k-1},\vich{k}{k-1}\} \\
      &=\la{}-(\LO{k-1}+\SH{k-1})+\{\max(\SH{k-1}),\max(\LO{k-1})\} \\
      &=\la{}-((\LO{k-1}-\max(\LO{k-1}))+(\SH{k-1}-\max(\SH{k-1}))) \\
      &=\la{}-(\LO{k}+\SH{k}),\end{array}\]
  with the second equality following by $\pro{\converge}{k-1}$. To prove $\pro{\pointers}{k}$, it suffices to show $\vic{k}\le\term$ if $\LO{k}$ is nonempty.  Suppose, for a contradiction, that $\LO{k}$ is nonempty and $\vic{k}>\term$.  Then
  \[\begin{array}{rll}
      n & = \sum_{i\in\Index{k}}\heia{k}{i} \\
        & = \partsum(\la{}-(\LO{k}+\SH{k})) + \thiefh{k}{k} &(\text{by }\pro{\converge}{k}) \\
        & \le n - \max(\LO{k}) + \flo{n/r} &\text{(since $\LO{k}$ is nonempty)} \\
        & < n &\text{(since $\max(\LO{k})>\flo{n/r}$)}.
      \end{array}\]
    This is a contradiction, so $\pro{\pointers}{k}$ follows.  Then $\pro{\ordinalgood}{k}$ follows from Lemma~\ref{graftlemma} and $\pro{\conditionx}{k-1}$.

    Finally, we prove $\pro{\tallimpliesshort}{k}$. Suppose that $\SH{k}$ is nonempty.  Suppose, for a proof by contradiction, that $\LO{k}$ is empty.  Then
  \[\begin{array}{rll} |\la{}-(\LO{k}+\SH{k})|+|\SH{k}|
                       & =|\la{}|=r=|\Index{k}| \\
                       & =|\{i \mid i\in\Index{k},i\notin[\vic{k},\term],i\neq\thief{k}\}|+|\{\thief{k}\}\cup[\vic{k},\term]| \\
                       & = |\{\heia{k}{i} \mid i\in\Index{k},i\notin[\vic{k},\term],i\neq\thief{k}\}|+|\{\thief{k}\}\cup[\vic{k},\term]| \\
                       & = |\la{}-(\LO{k}+\SH{k})|+|\{\thief{k}\}\cup[\vic{k},\term]|.
    \end{array}\]
  It follows that
  \begin{equation}\label{eqone}|\{\thief{k}\}\cup[\vic{k},\term]|=|\SH{k}|.\end{equation}
  Moreover, since
  \[\begin{array}{rll}
       \partsum(\la{}-(\LO{k}+\SH{k})) + \partsum(\LO{k}+\SH{k})
      &=\partsum(\la{}) = n = \sum_{i\in\Index{k}}\heia{k}{i}\\
      &=\sum_{i\in\Index{k},i\notin[\vic{k},\term],i\neq\thief{k}}\heia{k}{i}+\thiefh{k}{k}+\sum_{i=\vic{k}}^{\term}\heia{k}{i} \\
      & =\partsum(\la{}-(\LO{k}+\SH{k}))+\thiefh{k}{k}+\sum_{i=\vic{k}}^{\term}\heia{k}{i},
    \end{array}\]
  it follows that
  \begin{equation}\label{eqtwo}\partsum(\SH{k}) = \thiefh{k}{k}+\sum_{i=\vic{k}}^{\term}\heia{k}{i}.\end{equation}
  Note that $\heia{k}{\thief{k}}>\max(\SH{k})$ by $\pro{\heightgood}{k}$.  Also $\heia{k}{i}>\max(\SH{k})$ for all $i\in[\vic{k},\term]$.  Hence
  \[\begin{array}{rll}
      \max(\SH{k})\cdot |\SH{k}| & < \thiefh{k}{k} + \sum_{i=\vic{k}}^{\term}\heia{k}{i} &(\text{by (\ref{eqone})})\\
                                     & = \partsum(\SH{k}) &(\text{by (\ref{eqtwo})}) \\
      & \le\max(\SH{k})\cdot |\SH{k}|,
    \end{array}\]
a contradiction.  Proposition $\pro{\tallimpliesshort}{k}$ follows.

\subsection{Suppose Case 2 is satisfied at the start of iteration $k$.}\label{sec:case2} That is, suppose that $\SH{k-1}$ is empty and $\LO{k-1}$ is nonempty. Then the $k$th iteration of a primary loop is an iteration of Loop 2. Three of the propositional functions have simpler, logically equivalent formulations for this case:
  \begin{itemize}
  \item[$\lro{\multisetsdecrease}{k}:$] $\SH{k}$ is empty and $\LO{k}\subsetneq\LO{k-1}$.
  \item[$\lro{\converge}{k}:$] $\{\heia{k}{i} \mid i\in\Index{k}, i\notin [\vic{k},\term],i\neq\thief{k}\}=\la{}-\LO{k}$.
  \item[$\lro{\ordinalgood}{k}:$] If $\LO{k}$ is nonempty, then $\ordg{k}(\thief{k},\thiefh{k}{k})<\ordg{k}(\vic{k},\max\{1,\vich{k}{k}-(\max(\LO{k})-\thiefh{k}{k})+1\})$.
  \end{itemize}
  Observe that if $\SH{k}$ is empty, then $\lro{\heightgood}{k}$ and $\lro{\tallimpliesshort}{k}$ follow trivially.

The first step in executing the instruction set for Loop 2 is the execution of the \emph{Little Loop} (lines \ref{alg:main:littleloopstart}--\ref{alg:main:littleloopend}).  Let $\lgr{k-1}{\ell}$ (resp. $\lvic{k-1}{\ell}$) denote the state of the graph (resp. column pointer) at the end of the $\ell$th iteration of the Little Loop during the $k$th iteration, with $\lgr{k-1}{0}$ (resp. $\lvic{k-1}{0}$) denoting the initial state $\gr{k-1}$ (resp. $\vic{k-1}$).  Define the following propositional functions:
  \begin{itemize}
  \item[$\lpro{\lwelldef}{k}{\ell}:$] Each instruction is well-defined in iteration $\ell$ of the Little Loop during iteration $k$.
  \item[$\lpro{\lmodify}{k}{\ell}:$] $\matrixfxn(\lgr{k-1}{\ell})= \matrixfxn(\lgr{k-1}{\ell-1})+X$ for some strictly upper triangular matrix $X$.
  \item[$\lpro{\sciongrows}{k}{\ell}:$] $\lheia{k-1}{\ell}{\thief{k-1}}>\lheia{k-1}{\ell-1}{\thief{k-1}}$.
  \item[$\lpro{\lpropdw}{k}{\ell}:$] $\lgr{k-1}{l}$ is properly downward.
  \item[$\lpro{\lconverge}{k}{\ell}:$] $\{\lheia{k-1}{\ell}{i} \mid i\in\lIndex{k-1}{\ell},i\notin[\lvic{k-1}{\ell},\term],i\neq\thief{k-1}\}=\la{}-\LO{k-1}$.
  \item[$\lpro{\lconditionx}{k}{\ell}:$] If $i$ satisfies $\lvic{k-1}{\ell}\le i\le\term$, then:
    \begin{enumerate}
    \item $i\in\lIndex{k-1}{\ell}$;
    \item column $i$ is a downward path;
    \item $\lheia{k-1}{\ell}{i}=\flo{n/r}$ or $\lheia{k-1}{\ell}{i}=\cei{n/r}$;
    \item if $i<j\le\term$, then $\lheia{k-1}{\ell}{i}\le\lheia{k-1}{\ell}{j}$; and
    \item if $\LO{k}$ is nonempty, then $\lheia{k-1}{\ell}{i}<\max(\LO{k})$ for all $i\in[\lvic{k-1}{\ell},\term]$.
    \end{enumerate}
    The set of vertices $v$ with $\lvic{k-1}{\ell}\le x(v)\le\term$ are ordered by type-writer traversal.
  \item[$\lpro{\lpointers}{k}{\ell}:$] If $\lheia{k-1}{\ell}{\thief{k-1}}<\max(\LO{k-1})$, then $\lvic{k-1}{\ell}\le\term$.
  \item[$\lpro{\lordinalgood}{k}{\ell}:$] $\lordg{k-1}{\ell}(\thief{k-1},\lthiefh{k-1}{\ell}{k-1})$ \newline $<\lordg{k-1}{\ell}(\lvic{k-1}{\ell},\max\{1,\lvich{k-1}{\ell}{k-1}{\ell}-(\max(\LO{k})-\lthiefh{k-1}{\ell}{k-1})+1\})$.
  \end{itemize}
  We prove $\lpro{i}{k}{0}$ for all $i\in[4,8]$, and $\lpro{i}{k}{\ell}$ for all $i\in [1,8]$ and for all $\ell$ such that $\max(\LO{k-1})-\lheia{k-1}{\ell}{\thief{k-1}}>\cei{n/r}$. Propositions $\lpro{i}{k}{0}$ for $i\in[4,8]$ follow immediately from $\lro{i}{k-1}$ for $i\in[4,8]$, which are assumed as part of our inductive hypothesis. Let $\ell>0$ satisfy $\max(\LO{k-1})-\lheia{k-1}{\ell}{\thief{k-1}}>\cei{n/r}$, and suppose for an inductive hypothesis that $\lpro{i}{k}{\ell-1}$ holds for all $i\in [1,8]$.  We step through the two instructions (lines \ref{alg:main:littleloopfirst} and \ref{alg:main:littleloopend}) in the $\ell$th iteration of the Little Loop.  First, column $\lvic{k-1}{\ell-1}$ is grafted to column $\thief{k-1}$, to form the graph $\lgr{k-1}{\ell}$.  Since $\lgr{k-1}{\ell-1}$ is properly downward by $\lpro{\welldef}{k-1}{\ell-1}$ and $\thief{k-1}\neq\lvic{k-1}{\ell-1}$, it follows that the grafting transformation is well-defined. Next, $\lvic{k-1}{\ell}$ is set to $\lvic{k-1}{\ell-1}+1$, and iteration $\ell$ of the Little Loop concludes.

  Propositions $\lpro{i}{k}{\ell}$ for each $i\in[1,6]$ follow immediately, mostly as direct consequences of Lemma~\ref{graftlemma}.  To prove $\lpro{\lordinalgood}{k}{\ell}$, observe that Lemma~\ref{graftlemma}(\ref{graftlemmaordinal}) implies $\lordg{k-1}{\ell}(\thief{k-1},\lthiefh{k-1}{\ell}{k-1})=\lordg{k-1}{\ell-1}(\lvic{k-1}{\ell-1},\lheia{k-1}{\ell-1}{\lvic{k-1}{\ell-1}})$. But this is less than $\lordg{k-1}{\ell}(\lvic{k-1}{\ell},\max\{1,\lvich{k-1}{\ell}{k-1}{\ell}-(\max(\LO{k})-\lthiefh{k-1}{\ell}{k-1})+1\})$. Since $\lvic{k-1}{\ell-1}<\lvic{k-1}{\ell}$ and $\lheia{k-1}{\ell-1}{\lvic{k-1}{\ell-1}}\ge \lvich{k-1}{\ell}{k-1}{\ell}-(\max(\LO{k})-\lthiefh{k-1}{\ell}{k-1})$, the claim follows.  To prove $\lpro{\lpointers}{k}{\ell}$, suppose, for a contradiction, that $\lvic{k-1}{\ell}>\term$.  Then
\[\begin{array}{rll}
    n-\max(\LO{k-1})
    & < n-\lheia{k-1}{\ell}{\thief{k-1}} \\
    & = \sum_{i\in\lIndex{k-1}{\ell},i\neq\thief{k-1}}\lheia{k-1}{\ell}{i} \\
    & = \sum_{i\in\lIndex{k-1}{\ell},i\notin[\lvic{k-1}{\ell},\term],i\neq\thief{k-1}}\lheia{k-1}{\ell}{i} \\
    & = \partsum(\la{}-\LO{k-1}) &\text{(by $\lpro{\lconverge}{k}{\ell}$)} \\
    & \le n-\max(\LO{k-1}),
\end{array}\]
a contradiction. This concludes the proof of $\lpro{i}{k}{\ell}$ for all $i\in [1,8]$ and for all $\ell$ such that $\max(\LO{k-1})-\lheia{k-1}{\ell}{\thief{k-1}}>\cei{n/r}$.

Since $\lpro{\welldef}{k}{\ell}$ and $\lpro{\sciongrows}{k}{\ell}$ holds for all $\ell$ such that $\max(\LO{k-1})-\lheia{k-1}{\ell}{\thief{k-1}}>\cei{n/r}$, it follows that each iteration of the Little Loop is well-defined, and that the Little Loop eventually terminates.  Define $\lastlil$ to be $\max\{\ell \mid \max(\LO{k-1})-\lheia{k-1}{\ell-1}{\thief{k-1}}>\cei{n/r}\}$ if $\max(\LO{k-1})-\thiefh{k-1}{k-1}>\cei{n/r}$---i.e., if the Little Loop iterated at least once---and zero otherwise.

Exactly one of the following cases holds:
\begin{itemize}
\item[Case 2a:]$\max(\LO{k-1})-\thiefh{k-1}{k-1}>\lheia{k-1}{\lastlil}{\lvic{k-1}{\lastlil}}$ and $\lheia{k-1}{\lastlil}{\lvic{k-1}{\lastlil}+1}=\cei{n/r}$;
\item[Case 2b:]$\max(\LO{k-1})-\thiefh{k-1}{k-1}>\lheia{k-1}{\lastlil}{\lvic{k-1}{\lastlil}}$ and $\lheia{k-1}{\lastlil}{\lvic{k-1}{\lastlil}+1}=\flo{n/r}$;
\item[Case 2c:]$\max(\LO{k-1})-\thiefh{k-1}{k-1}=\lheia{k-1}{\lastlil}{\lvic{k-1}{\lastlil}}$; and
\item[Case 2d:]$\max(\LO{k-1})-\thiefh{k-1}{k-1}<\lheia{k-1}{\lastlil}{\lvic{k-1}{\lastlil}}$.
\end{itemize}

To verify that the conditional expressions for Cases 2a and 2b are well-defined, we show that $\lvic{k-1}{\lastlil}+1\in\lIndex{k-1}{\lastlil}$ if 
\begin{equation}\label{distinctivelabel}\max(\LO{k-1})-\thiefh{k-1}{k-1}>\lheia{k-1}{\lastlil}{\lvic{k-1}{\lastlil}}.\end{equation}  Assume that (\ref{distinctivelabel}) holds.  Suppose, for a contradiction, that $\lvic{k-1}{\lastlil}+1>\term$.  Then $\lpro{\lconditionx}{k-1}{\lastlil}$ implies that $\lvic{k-1}{\lastlil}=\term$.  Hence
\[\begin{array}{rll}
    n &=\sum_{i\in\lIndex{k-1}{\lastlil},i\notin[\lvic{k-1}{\lastlil},\term],i\neq\thief{k-1}}\lheia{k-1}{\lastlil}{i} \\
      &\hspace{1cm}+\lheia{k-1}{\lastlil}{\lvic{k-1}{\lastlil}} + \lheia{k-1}{\lastlil}{\thief{k-1}} \\
      &=\partsum(\la{}-\LO{k-1})+\lheia{k-1}{\lastlil}{\lvic{k-1}{\lastlil}} + \lheia{k-1}{\lastlil}{\thief{k-1}} & \text{(by $\lpro{\lconverge}{k}{\lastlil}$)}\\
      &<\partsum(\la{}-\LO{k-1})+\max(\LO{k-1}) & \text{(by (\ref{distinctivelabel}))}\\
      &\le n,
  \end{array}\]
a contradiction. The claim follows.

A different set of instructions executes for each case. We prove the inductive step for Case 2a in Section~\ref{twoa}; Cases 2b--d can be proved similarly.

\subsubsection{Suppose Case 2a holds.}\label{twoa}  Then $\lheia{k-1}{\lastlil}{\lvic{k-1}{\lastlil}}=\flo{n/r}$, $\lheia{k-1}{\lastlil}{\lvic{k-1}{\lastlil}+1}=\cei{n/r}$, and $\max(\LO{k-1})-\thiefh{k-1}{k-1}=\cei{n/r}$.  The Case 2a instruction set is executed.  The graph $\gr{k}$ is formed by grafting column $\lvic{k-1}{\lastlil}+1$ in $\lgr{k-1}{\lastlil}$ to $\thief{k-1}$. Recall from the discussion immediately preceding this section that $\lvic{k-1}{\lastlil}+1\le\term$. Hence column $\lvic{k-1}{\lastlil}+1$ is a downward path by $\lpro{\lconditionx}{k-1}{\lastlil}$.  Since $\lgr{k-1}{\lastlil}$ is properly downward by $\lpro{\lpropdw}{k-1}{\lastlil}$, and since column $\lvic{k-1}{\lastlil}+1$ is a downward path, it follows that the grafting operation is well-defined.  Lemma~\ref{graftlemma}(\ref{graftlemmaheight}) implies that $\thiefh{k}{k-1}=\max(\LO{k-1})$. The next instructions set $\thief{k}=\lvic{k-1}{\lastlil}$, $\vic{k}=\lvic{k-1}{\lastlil}+1$, and $\LO{k}=\LO{k-1}-\{\max(\LO{k-1})\}$.

Proposition $\lro{\welldef}{k}$ follows since each of the expressions and instructions above are well-defined, and since $\lpro{\lwelldef}{\ell}{k}$ holds for all $\ell$ such that $\max(\LO{k-1})-\lheia{k-1}{\ell}{\thief{k-1}}>\cei{n/r}$.  Proposition $\lro{\multisetsdecrease}{k}$ is immediately clear. Proposition $\lro{\propdw}{k}$ follows from Lemma~\ref{graftlemma}(\ref{graftlemmapropdw}). Since $\lpro{\lmodify}{k}{\ell}$ holds for all $\ell$ such that $\max(\LO{k-1})-\lheia{k-1}{\ell}{\thief{k-1}}>\cei{n/r}$, it follows that $\lro{\modify}{k}$.  It is straight-forward to verify that $\lro{\conditionx}{k}$ follows from Lemma~\ref{graftlemma} and $\lpro{\lconditionx}{k}{\lastlil}$.  Proposition $\lro{\ordinalgood}{k}$ is a straight-forward consequence of $\lpro{\lconditionx}{k}{\lastlil}$ and Lemma~\ref{graftlemma}(\ref{graftlemmaordinal}).

To prove $\lro{\converge}{k}$, first observe that Lemma~\ref{graftlemma}(\ref{graftlemmadomain}) implies that
\[\begin{array}{rll}
    & \{i \mid i\in\Index{k},i\notin[\vic{k},\term],i\neq\thief{k}\} \\
    &=\{i \mid i\in\lIndex{k-1}{\lastlil}-\{\lvic{k-1}{\lastlil}+1\},i\notin[\lvic{k-1}{\lastlil}+2,\term],i\neq\lvic{k-1}{\lastlil}\} \\
    &=\{i \mid i\in\lIndex{k-1}{\lastlil},i\notin[\lvic{k-1}{\lastlil},\term]\}. \\
  \end{array}\]
Hence
\[\begin{array}{rll}
    & \{\heia{k}{i} \mid i\in\Index{k},i\notin[\vic{k},\term],i\neq\thief{k}\} \\
    &=\{\heia{k}{i} \mid i\in\lIndex{k-1}{\lastlil},i\notin[\lvic{k-1}{\lastlil},\term]\} \\
    &=\{\heia{k}{i} \mid i\in\lIndex{k-1}{\lastlil},i\notin[\lvic{k-1}{\lastlil},\term],i\neq\thief{k-1}\}+\{\heia{k}{\thief{k-1}}\} \\
    &=(\la{}-\LO{k-1})+\{\max(\LO{k-1})\} \hspace{1cm}\text{(by $\lpro{\lconverge}{k}{\lastlil}$)}\\
    &=\la{}-\LO{k}.
  \end{array}\]
Proposition $\lro{\converge}{k}$ follows.

Finally, we prove $\lro{\pointers}{k}$.  Assume that $\LO{k}$ is nonempty.  Suppose, for a contradiction, that $\vic{k}>\term$.  Then
\[\begin{array}{rll}
    n-\flo{n/r}
    &=\sum_{i\in\Index{k}}\heia{k}{i}-\thiefh{k}{k} \\
    &=\sum_{i\in\Index{k},i\notin[\vic{k},\term],i\neq\thief{k}}\heia{k}{i} \\
    &=\partsum(\la{}-\LO{k}) & \text{(by $\lro{\converge}{k}$)} \\
    &\le n-\max(\LO{k}),
  \end{array}\]
a contradiction since $\max(\LO{k})>\flo{n/r}$.  Proposition $\lro{\pointers}{k}$ follows.

\bibliography{references}
\bibliographystyle{amsplain}

\end{document}